\documentclass{amsart}
\usepackage{amsmath, amssymb, amsthm}
\usepackage{amscd}
\usepackage{graphicx} 
\usepackage{enumerate}

\def\CC{{\mathbb C}}

\def\QQ{{\mathbb Q}}
\def\PP{{\mathbb P}}
\def\QQ{{\mathbb Q}}
\def\RR{{\mathbb R}}

\def\Qbar{\overline{\mathbb Q}}

\def\0{{\mathbf 0}}
\def\1{{\mathbf 1}}

\def\Ocal{{\mathcal O}}

\def\Kbar{{\bar K}}

\def\Gal{\mathrm{Gal}}

\def\PGL{\mathrm{PGL}}

\def\Berk{\mathrm{Berk}}
\def\supp{\mathrm{supp}}

\title[Energy integrals over local fields]{Energy integrals over local fields \\ and global height bounds}
 
\author[Fili]{Paul Fili}
\address{Department of Mathematics\\ Oklahoma State University, Stillwater, OK 74078}
\email{fili@post.harvard.edu}
\author[Petsche]{Clayton Petsche}
\address{Department of Mathematics\\ Oregon State University, Corvallis, OR 97331}
\email{petschec@math.oregonstate.edu}

\date{\today}

\usepackage[pdftitle={Energy integrals over local fields},pdfauthor={Fili and Petsche},pdfstartview={}]{hyperref}

\usepackage[all]{xy}

\newtheorem{thm}{Theorem}

\newtheorem{cor}[thm]{Corollary}

\newtheorem*{thm*}{Theorem}
\newtheorem*{alg*}{Algorithm}
\newtheorem*{lemma*}{Lemma}

\theoremstyle{remark}

\newtheorem*{rmk*}{Remark}

\newtheorem*{notation*}{Notation}

\newtheorem*{example*}{Example}

\theoremstyle{definition}

\newtheorem*{defn*}{Definition}

\newcommand{\mybf}{\mathbb}

\newcommand{\bP}{\mybf{P}}
\newcommand{\bR}{\mybf{R}}

\newcommand{\bC}{\mybf{C}}
\newcommand{\bN}{\mybf{N}}

\newcommand{\bQ}{\mybf{Q}}
\newcommand{\bF}{\mybf{F}}

\newcommand{\al}{\alpha}

\providecommand{\abs}[1]{\lvert#1\rvert}
\providecommand{\norm}[1]{\lVert#1\rVert}

\newcommand{\ddx}{\frac{d}{dx}}

\newcommand{\ra}{\rightarrow}

\def\talltareesidedbox#1{\setbox0=\hbox{$#1$}\dimen0=\wd0 \advance\dimen0 by3pt\rlap{\hbox{\vrule height10pt width.4pt
 depth2pt \kern-.4pt\vrule height10.4pt width\dimen0 depth-10pt\kern-.4pt \vrule height10pt width.4pt depth2pt}}
 \relax \hbox to\dimen0{\hss$#1$\hss}}
\def\tareesidedbox#1{\setbox0=\hbox{$#1$}\dimen0=\wd0 \advance\dimen0 by3pt\rlap{\hbox{\vrule height8pt width.4pt
 depth2pt \kern-.4pt\vrule height8.4pt width\dimen0 depth-8pt\kern-.4pt \vrule height8pt width.4pt depth2pt}}
\relax \hbox to\dimen0{\hss$#1$\hss}}
\def\shorttareesidedbox#1{\setbox0=\hbox{$#1$}\dimen0=\wd0 \advance\dimen0 by3pt\rlap{\hbox{\vrule height7pt width.4pt
 depth2pt \kern-.4pt\vrule height7.4pt width\dimen0 depth-7pt\kern-.4pt \vrule height7pt width.4pt depth2pt}}
 \relax \hbox to\dimen0{\hss$#1$\hss}}

\begin{document}

\begin{abstract}
 We solve an energy minimization problem for local fields. As an application of these results, we improve on lower bounds set by Bombieri and Zannier for the limit infimum of the Weil height in fields of totally $p$-adic numbers and generalizations thereof. In the case of fields with mixed archimedean and non-archimedean splitting conditions, we are able to combine our bounds with similar bounds at the archimedean places for totally real fields.
\end{abstract}

\thanks{The second author was supported in part by NSF grant DMS-0901147. The authors would like to thank Igor Pritsker for bringing \cite{Simeonov} to our attention.}

\maketitle

\section{Introduction}
\subsection{Local energy minimization and equidistribution}  Let $L$ be a complete field with respect to a nontrivial absolute value $|\cdot|$.  Given a Borel probability measure $\nu$ on $\PP^1(L)$, define the energy integral  
\begin{equation}\label{EnergyIntegral}
I(\nu)=\iint_{\PP^1(L)\times\PP^1(L)}-\log\delta(x,y)\,d\nu(x)\,d\nu(y),
\end{equation}
where $\delta:\PP^1(L)\times\PP^1(L)\to\RR$ is defined by 
\begin{equation*}
\delta(x,y) =\frac{|x_0y_1-y_0x_1|}{\max\{|x_0|,|x_1|\}\max\{|y_0|,|y_1|\}}
\end{equation*}
for $x=(x_0:x_1)$ and $ y=(y_0:y_1)$ in $\PP^1(L)$.  The function $\delta:\PP^1(L)\times\PP^1(L)\to\RR$ is symmetric, nonnegative, continuous, and vanishes precisely along the diagonal $\Delta$ of $\PP^1(L)\times\PP^1(L)$, and thus $-\log\delta(x,y)$ is a reasonable choice for a potential kernel on $\PP^1(L)$.  In the non-archimedean case, $\delta(x,y)$ is the standard projective metric on $\PP^1(L)$, and in both the non-archimedean and archimedean cases, $-\log\delta(x,y)$ may be viewed as a continuously varying family of Weil local height functions.

The first main result of this paper is the solution to the minimization problem for the energy integral $(\ref{EnergyIntegral})$ in the case where the field $L$ is locally compact but not algebraically closed.

\begin{thm}\label{MainThmLocal}
Let $L$ be a non-algebraically-closed (equivalently, non-complex) locally compact field with respect to a nontrivial absolute value $|\cdot|$, which is normalized to coincide with the modulus of additive Haar measure on $L$.  
\begin{itemize}
	\item[{\bf (a)}] There exists a unique Borel probability measure $\mu_L$ on $\PP^1(L)$ such that $I(\nu)\geq I(\mu_L)$ for all Borel probability measures $\nu$ on $\PP^1(L)$.  If $\{\nu_n\}$ is a sequence of Borel probability measures on $\PP^1(L)$ with $I(\nu_n)\to I(\mu_L)$, then $\nu_n\to\mu_{L}$ weakly.    
	\item[{\bf (b)}]  In the real case, the measure $\mu_\RR$ of minimal energy is given explicitly by
\begin{equation*}
\mu_\RR(x)=\frac{1}{\pi^2x}\log\bigg|\frac{x+1}{x-1}\bigg| \,dx
\end{equation*}
and the minimal energy is
\begin{equation*}
I(\mu_\RR)=\frac{7\zeta(3)}{2\pi^2}=0.426278\ldots 
\end{equation*}
where $\zeta(3)=\sum_{n\geq1}n^{-3}$.
	\item[{\bf (c)}]  In the non-archimedean case, the measure $\mu_L$ of minimal energy is the unique $\PGL_2(\Ocal_L)$-invariant Borel probability measure on $\PP^1(L)$, and the minimal energy is
\begin{equation*}
I(\mu_L)=\frac{q\log q}{q^2-1}
\end{equation*}
where $q$ denotes the number of elements in the residue field of $L$.
\end{itemize}
\end{thm}

Our inspiration for Theorem~\ref{MainThmLocal} comes from the solution in the algebraically closed setting, which is due to M. Baker and R. Rumely.  First, when $L=\CC$, the relevant result is implicit in Rumely \cite{RumelyArticle}, and in a formulation closer to the one considered here in Baker-Rumely\footnote{In fact, the results of Baker-Rumely hold, in both the complex and non-archimedean cases, for a much more general class of potential kernels arising from the study of the dynamics of rational maps on $\PP^1$.  For our purposes it is enough to confine our discussion of their results to this special case.} \cite[\S 3.6]{BakerRumely}.  The unique Borel probability measure $\mu_\CC$ on $\PP^1(\CC)$ minimizing the energy integral $(\ref{EnergyIntegral})$ is the normalized Haar measure supported on the unit circle of $\CC$, and the minimal energy $I(\mu_\CC)$ is zero.  Among other things, Rumely's result in \cite{RumelyArticle} provides a new potential-theoretic proof of the Bilu equidistribution theorem \cite{Bilu} on global points of small Weil height.

In \cite{BakerRumely} (and \cite[\S 10.2]{BakerRumelyBook}), Baker-Rumely also treat the case in which $L$ is non-archimedean and algebraically closed. In the non-archimedean setting, if $L$ is algebraically closed it follows that $\PP^1(L)$ is not compact, which is problematic from the point of view of minimizing the energy integral $(\ref{EnergyIntegral})$.  Indeed, while it follows from the nonnegativity of $-\log\delta(x,y)$ in the non-archimedean setting that $I(\nu)\geq0$ for all Borel probability measures $\nu$ on $\PP^1(L)$, it may also be shown when $L=\bC_p$ that no case of equality exists, and one can find sequences $\{\nu_n\}$ of Borel probability measures on $\PP^1(L)$, possessing no weak limit, but for which $I(\nu_n)\to0$.  Baker-Rumely overcome this difficulty by passing from the ordinary projective line $\PP^1(L)$ to the (compact) Berkovich projective line $\PP^1_{\Berk}(L)$, and they show that the unique Borel probability measure $\mu_L$ on $\PP^1_{\Berk}(L)$ of minimal energy is the Dirac measure supported at the Gauss point $\zeta$ of $\PP^1_{\Berk}(L)$.

Strictly speaking, Theorem~\ref{MainThmLocal} holds vacuously for measures $\nu$ charging any single point of $\PP^1(L)$, because any such measure has energy $I(\nu)=+\infty$, and thus the hypothesis $I(\nu_n)\to I(\mu_L)$ cannot be satisfied along sequences of such measures.  However, with some care it is possible to obtain a nontrivial discrete analogue of Theorem~\ref{MainThmLocal} as a corollary.  Given a finite subset $Z$ of $\PP^1(L)$ with $|Z|=N$, define the {\em energy sum} (or {\em discrepancy}) of the set $Z$ by
\[
D(Z)=\frac{1}{N(N-1)}\sum_{\stackrel{\alpha,\beta\in Z}{\alpha \neq \beta}}-\log\delta(\alpha,\beta).
\]

\begin{cor}\label{DiscreteCor}
Let $L$ be a non-algebraically-closed (equivalently, non-complex) locally compact field with respect to a nontrivial absolute value $|\cdot|$.  If $\{Z_n\}$ is a sequence of finite subsets of $\PP^1(L)$ with $|Z_n|\to\infty$, then 
\begin{equation}\label{DiscreteLimInf}
\liminf_{n\to\infty}D(Z_n)\geq I(\mu_L).
\end{equation}
Moreover, if $D(Z_n)\to I(\mu_L)$, then the sequence $\{Z_n\}$ is $\mu_L$-equidistributed in the sense that $[Z_n]\to\mu_L$ weakly, where $[Z_n]$ denotes the unit Borel measure on $\PP^1(L)$ supported equally on the points of $Z_n$.
\end{cor}

\subsection{Global height bounds}
The most obvious difference between the locally compact, non-complex setting of Theorem~\ref{MainThmLocal} and Corollary~\ref{DiscreteCor}, and the results of Baker-Rumely in the algebraically closed case, is that the minimal energies are positive in the former case, and zero in the latter.  The second main result of this paper makes use of this positivity to obtain lower bounds on the usual absolute Weil height $h:\Kbar\to\RR$ for algebraic numbers over a number field $K$ satisfying certain splitting conditions.  This connection is made possible by the dual role of the function $-\log\delta(x,y)$ as the potential kernel for the energy integral $\eqref{EnergyIntegral}$ on the one hand, and as a continuously varying family of Weil local height functions on the other hand.

The following theorem illustrates our main global result in the simplest special case of algebraic numbers over $\QQ$ with mixed splitting conditions.  Denote by $M_\QQ=\{\infty,2,3,5,\dots\}$ the set of places of $\QQ$.  An algebraic number $\alpha\in \Qbar$ is said to be {\em totally real} (resp. {\em totally $p$-adic}) if its minimal polynomial over $\QQ$ splits completely over $\RR$ (resp. $\QQ_p$). 

\begin{thm}\label{SimpleGlobalThm}
Let $S$ be a nonempty subset of $M_\QQ$, and let $L_S$ be the subfield of $\Qbar$ consisting of all those $\alpha\in\Qbar$ such that $\alpha$ is totally $p$-adic for all primes $p\in S$, and $\alpha$ is totally real if $\infty\in S$.  Then
\begin{equation}\label{SimpleGlobalThmBound}
\liminf_{\alpha\in L_S}h(\alpha)\geq
\begin{cases}
\displaystyle \frac{1}{2} \sum_{p\in S} \frac{p\log p}{p^2-1} & \text{if }\infty\not\in S\\
\displaystyle \frac{1}{2} \sum_{\substack{p\in S\\ p\nmid\infty}} \frac{p\log p}{p^2-1}  + \frac{7\zeta(3)}{4\pi^2} & \text{if }\infty\in S
\end{cases}
\end{equation}
\end{thm}

In all cases except $S=\{\infty\}$, the lower bound in Theorem~\ref{SimpleGlobalThm} is the best currently known.  Our result was inspired in part by work of Bombieri-Zannier \cite{BombieriZannierNote}, who used an elementary counting argument to study the distribution of the $\Gal(\Qbar/\QQ)$-conjugates of $\alpha$ in the residue classes modulo primes lying above $p$.  They treat only the case $\infty\notin S$, and in the same setting as Theorem~\ref{SimpleGlobalThm} their bound is 
\begin{equation*}
\liminf_{\alpha\in L_S}h(\alpha)\geq\frac{1}{2} \sum_{p\in S} \frac{\log p}{p+1},
\end{equation*}
which is slightly worse than Theorem~\ref{SimpleGlobalThm}.  While very similar in spirit to the approach of Bombieri-Zannier, our potential-theoretic approach allows us to treat the real place on equal footing with the non-archimedean places as well as leading to an improvement in the bound at each finite prime $p\in S$.

In the case $S=\{\infty\}$, it was shown by Schinzel that $$
h(\alpha)\geq\frac{1}{2}\log\bigg(\frac{1+\sqrt{5}}{2}\bigg)=0.24061...
$$ 
whenever $\alpha\in\Qbar\setminus\{0,\pm1\}$ is totally real, and this bound is better than then right-hand-side of $(\ref{SimpleGlobalThmBound})$, which is $7\zeta(3)/4\pi^2=0.21314...$. 

While our archimedean result is weaker than this global bound, it has the advantage that it can be used in conjunction with similar restrictions on splitting at the non-archimedean primes, and thus we can achieve bounds in the cases of mixed archimedean and non-archimedean splitting conditions.  This is illustrated in the following example.  

\begin{example*}
Let $S=\{2,\infty\}$, so that in the notation of the theorem, $L_S$ is the field of all numbers which are both totally $2$-adic and totally real. Then Theorem \ref{SimpleGlobalThm} implies that
\[
 \liminf_{\al\in L_S} h(\al) \geq \frac{1}{2}\cdot\frac{2\log 2}{2^2-1} + \frac{7\zeta(3)}{4\pi^2} = 0.231049\ldots + 0.213139\ldots = 0.444188\ldots 
\]
In contrast, if we had applied the Bombieri-Zannier bound separately at $p=2$, we could say that
\[
 \liminf_{\al\in L_S} h(\al)\geq \frac{1}{2} \cdot \frac{\log 2}{2+1} = 0.115525\ldots
\]
and if we had applied the bound of Schinzel for totally real numbers, we could say that
\[
  \liminf_{\al\in L_S} h(\al)\geq \frac{1}{2}\log\bigg(\frac{1+\sqrt{5}}{2}\bigg)=0.24061...
\]
\end{example*}

We note that a similar potential-theoretic approach was taken by the first author in \cite{F-tot-p} to analyze the case of totally $p$-adic {\em algebraic integers}, arriving at the bound of
\begin{equation*}
\liminf_{\alpha\in \Ocal_{L_S}}h(\alpha)\geq\frac{1}{2} \sum_{p\in S} \frac{\log p}{p-1},
\end{equation*}
where $\Ocal_{L_S}$ denotes the ring of algebraic integers of $L_S$.

Theorem \ref{SimpleGlobalThm} and Corollary \ref{DiscreteCor} together imply the following global equidistribution result:
\begin{cor}\label{cor:global-equid}
 Let $S$ be a nonempty subset of $M_\QQ$, and let $L_S$ be the subfield of $\Qbar$ consisting of all those $\alpha\in\Qbar$ such that $\alpha$ is totally $p$-adic for all primes $p\in S$, and $\alpha$ is totally real if $\infty\in S$.  If $\al_n\in L_S$ is a sequence such that
\begin{equation*}
\lim_{n\ra\infty }h(\alpha_n)=
\frac{1}{2} \sum_{\substack{p\in S\\ p\nmid\infty}} \frac{p\log p}{p^2-1}  + \sum_{\substack{p\in S\\ p\mid\infty}} \frac{7\zeta(3)}{4\pi^2},
\end{equation*}
then for each $p\in S$, the probability measures $[\al_n]$ supported equally on each Galois conjugates of $\al_n$ converge weakly to the measure $\mu_{\bQ_p}$.
\end{cor}
Essentially, this result states that if such a sequence exists, then the $p$-adic conjugates of each $\al_n$ must equidistribute along $\bP^1(\bQ_p)$ for each $p\in S$ according the measure determined in Theorem \ref{MainThmLocal}. We note that as a coarse but interesting corollary of this result, the conjugates of each $\al_n$ must also distribute uniformly amongst the residue classes of $\bP^1(\bF_p)$ for each $p\in S$. It is worth observing that if $S=\{\infty\}$, then the above theorem is vacuous, as by Schinzel's result the lower bound is not attained. In all other cases, however, it is unknown if the lower bound on the height is achieved. We remark that the bound in Theorem \ref{SimpleGlobalThm} is of the correct order of magnitude, as it follows from a more general result in \cite[Theorem 2]{F-tot-p} that, for $S$ which do not include $\infty$, we have
\[
 \liminf_{\al\in L_S} h(\al) \leq \sum_{p\in S} \frac{\log p}{p-1}.
\]

\section{Local energy minimization and equidistribution}

\subsection{Archimedean preliminaries}
In this section we will prove the basic potential theoretic results for the $-\log \delta(x,y)$ kernel in the case where $L$ is archimedean and non-complex, that is, when $L=\bR$. As it makes no difference in our proofs, we will prove these results for arbitrary closed subsets of $\bP^1(\bC)$. Let $\abs{\cdot}$ be the usual absolute value on $\bC$, and for a Borel probability measure $\mu$ on $\bP^1(\bC)$, define the \emph{energy} to be
\begin{equation}
 I(\mu) = \iint_{\bP^1(\bC)^2} -\log \delta(x,y)\,d\mu(x)\,d\mu(y).
\end{equation}
When $\mu$ is supported on $L=\bR$, this is the same energy integral defined in \eqref{EnergyIntegral}. We can associate a potential function to $\mu$ by:
\begin{equation}
  p_\mu(x) = \int_{\bP^1(\bC)} -\log \delta(x,y)\,d\mu(y).
\end{equation}
Notice that $-\log 2 \leq -\log \delta(x,y)\leq \infty$ for all $x,y\in \bP^1(\bC)$, so the energy $I(\mu)$ always exists as a value in $[-\log 2,\infty]$.

We will begin by proving the standard results of potential theory for the energy kernel in this setting.
\begin{thm}[Maria's theorem]
 Let $\mu$ be a Borel probability measure on $\bP^1(\bC)$ of finite energy and let $\mu$ be supported on the closed set $E\subset \bP^1(\bC)$. If the potential satisfies
 \[
  p_\mu(x) \leq M<\infty \quad\text{for all}\quad x\in E,
 \]
 then in fact
 \[
 p_\mu(x)\leq M \quad\text{for all}\quad x\in \bP^1(\bC).
 \]
\end{thm}
\begin{proof}
 Let $U$ be the complement of $E$ in the Riemann sphere. Then for $y\in E$, it follows that $-\log \delta(x,y)$ is a subharmonic function of $x$, since for $x\neq\infty$, $-\log \delta(x,y) = \max\{\log |x|,0\} +\max\{\log |y|,0\} + \log \abs{1/(x-y)}$ and $\log \abs{f}$ is subharmonic when $f$ is holomorphic, and a sum or maximum of subharmonic functions is again subharmonic, and when $x=\infty$, we simply note that in a neighborhood of $\infty$ we have $-\log \delta(x,y) = \max\{\log |y|,0\} + \log \abs{x/(x-y)}$ so this is again subharmonic. It follows by \cite[Theorem 2.4.8]{Ransford} that $p_\mu(x)$ is then subharmonic on $U$, and so Maria's theorem follows from the maximum principle for subharmonic functions \cite[Theorem 2.3.1]{Ransford}.
\end{proof}

Before stating Frostman's theorem, we recall that a set $F$ is called \emph{polar} if it has logarithmic capacity 0, or equivalently, if every Borel probability measure $\nu$ supported on $F$ has infinite energy with respect to the usual logarithmic kernel. It is easy to see that the collection of polar sets in $\bC$ is the same for the energy with respect to the usual logarithmic kernel $-\log\abs{x-y}$ as they are for the energy defined above.
\begin{thm}[Frostman's theorem]\label{thm:arch-frostman}
 Let $E$ be a closed subset of $\bP^1(\bC)$, and assume that there exists some Borel probability measure supported on $E$ with finite energy. Then there exists a unique Borel probability measure $\mu$ on $E$ such that
 \[
  I(\mu) = \inf_{\substack{\nu\\ \supp(\nu)\subseteq E}} I(\nu) < \infty,
 \]
 where the infimum is taken over all Borel probability measures supported on $E$. We call this measure $\mu$ the \emph{equilibrium measure of $E$} and we call $V=V(E)=I(\mu)$ be the \emph{Robin constant of $E$}. Further, the associated potential $p_\mu$ satisfies
 \[
  p_\mu(x)\leq I(\mu)\quad\text{for all}\quad x\in\bP^1(\bC)
 \]
 and $p_\mu(x)=I(\mu)$ for all $x\in E\setminus F$, where $F$ is a polar subset of $E$.
\end{thm}
\begin{proof}
This result follows from Theorem 1.1 of \cite{Simeonov} with the choice of weight $w(z)=1/\max\{\abs{z},1\}$.
%
%
\end{proof}

\begin{thm}
 Let $E\subset\bP^1(\bC)$ be a closed set of finite energy. Let $\mu$ be the equilibrium measure of $E$ and let $V=I(\mu)$ be the Robin constant. Then for any Borel probability measure $\nu$ supported on $E$,
 \[
  \inf_{x\in E} p_\nu(x) \leq V \leq \sup_{x\in E} p_\nu(x).
 \]
\end{thm}
\begin{proof}
 Our proof largely follows \cite[Theorem III.15]{Tsuji}. Since $-\log \delta(x,y)\geq -\log 2$ on $\bP^1(\bC)$ and $\mu,\nu$ are probability measures, it follows from Tonelli's theorem that
 \[
  \int_E p_\nu(x) \,d\mu(x) = \int_E p_\mu(x)\,d\nu(x).
 \]
 Since $p_\mu(x) \leq V$ everywhere, it follows that
 \[
  \int_E p_\nu(x) \,d\mu(x)\leq V
 \]
 hence $\inf_{x\in E} p_\nu(x) \leq V$. In other direction, we may as well assume that that $p_\nu(x)<\infty$, and from the usual maximum principle argument it follows that $I(\nu)<\infty$. It then follows from the proof of \cite[Theorem III.7]{Tsuji}, the only change being replacing the potential kernel with the $-\log \delta(x,y)$, that $\nu$ cannot assign any positive mass to a set of capacity zero. Therefore, since $p_\mu(x) = V$ for all $x\in E$ except possibly on a set of capacity zero, it follows that
 \[
  \int_E p_\mu(x)\,d\nu(x) = V,
 \]
 and hence we can conclude that $\sup_{x\in E} p_\nu(x) \geq V$.
\end{proof}
From the the uniqueness of the equilibrium measure and the above theorem we immediately gain the following result:
\begin{cor}\label{cor:constant-potential-is-equil}
 If $\nu$ is a Borel probability measure supported on a closed set $E\subset\bP^1(\bC)$ and $p_\nu(x)=C$ for all $x\in E\setminus F$, where $F$ is a polar subset, then $\nu$ is the equilibrium measure of $E$ and $V(E)=C$.
\end{cor}

\subsection{The minimal energy}
In this section we will prove Theorem~\ref{MainThmLocal} and Corollary~\ref{DiscreteCor}.
\begin{proof}[Proof of Theorem \ref{MainThmLocal} (a)]
 For this part of the problem, we only assume $L$ is local and non-complex, as our proof will work in both the archimedean and non-archimedean settings. First, we may suppose that
\[
 V(L) = \inf_{\nu} I(\nu) < \infty
\]
 where the infimum is taken over Borel probability measures supported on $\bP^1(L)$, as we will construct measures of finite energy in (b) and (c). In the non-archimedean setting, notice that $\delta(x,y)$ as defined above in fact coincides with the Hsia kernel on $\PP_\Berk^1(F)\times \PP_\Berk^1(F)$ with respect to the basepoint given by the Gauss point $\zeta=\zeta_{0,1}$. The existence and uniqueness of the minimal measure is then nothing more than Frostman's theorem, which in the non-archimedean setting is proven in \cite[Theorem 6.18, Corollary 7.21]{BakerRumelyBook} for the potential with respect to the Gauss point $\zeta$ of $F$, noting that $\bP^1(L)$ is a compact subset of $\PP_\Berk^1(F)\setminus \{\zeta\}$, and in the archimedean setting follows from Theorem \ref{thm:arch-frostman}.
 
 Now suppose we have a sequence Borel probability measures $\{\nu_n\}$ supported on $\bP^1(L)$ such that
 \[
  I(\nu_n) \ra V(L). 
 \]
 We will show that $\nu_n\ra \mu_L$ weakly by demonstrating that every subsequence has a further subsequence which converges weakly to $\mu_L$. Given any subsequence of $\nu_n$, we can assume by passing to a further subsequence $\nu_{n_k}$ that the sequence converges weakly to some Borel probability measure $\nu$ on $\bP^1(L)$, since the space of such measures is weak-* compact. Since $-\log \delta(x,y)$ is lower semicontinuous and bounded below, it follows by the same argument as in \cite[Proposition 6.6]{BakerRumelyBook} that
\[
 I(\nu) = \lim_{k\ra \infty} I(\nu_{n_k}) = V(L).
\]
The uniqueness of the equilibrium measure from Frostman's theorem then implies that $\nu=\mu_L$, proving the claim.
\end{proof}

\begin{proof}[Proof of Theorem \ref{MainThmLocal} (b)]
 Let us start by proving that $\mu = \mu_\bR$ as defined by
\[
 d\mu(x) = \frac{1}{\pi^2 x} \log \left|\frac{x+1}{x-1}\right|\,dx  
\]
 is in fact a Borel probability measure on $\bP^1(\bR)$. It is absolutely continuous with respect to the Lebesgue measure of $\bR$ with nonnegative derivative so it suffices to show that it has total mass $1$. We compute, taking advantage of the fact that the integrand is even and invariant under the change of variables $x\mapsto 1/x$,
\begin{align*}
 \int_{-\infty}^{\infty} d\mu(x) &= 4\int_0^1 \frac{1}{\pi^2 x} \log\frac{1+x}{1-x}\,dx\\
&= \frac{8}{\pi^2} \int_0^1 \frac{1}{x} \sum_{n=1}^\infty \frac{x^{2n-1}}{2n-1}\,dx\\
&= \frac{8}{\pi^2} \sum_{n=1}^\infty \frac{1}{(2n-1)^2} = 1.
\end{align*}
 By the above Corollary \ref{cor:constant-potential-is-equil}, it suffices to show that the associated potential $p_\mu(x)$ is finite and constant on $\bP^1(\bR)\setminus F$ for a polar set $F$, and the result will follow. Recall that the Hilbert transform (see \cite[\S III.1]{SteinSingularIntegrals}) of $f:\bR\ra \bR$ is defined to be
 \[
  \tilde{f}(x) = \frac{1}{\pi} \int_{\bR} \frac{f(t)}{x-t}\,dt
 \]
 with the integral extended over the singularity by the principal value. As is well-known, the Hilbert transform of a function in $L^p(\bR)$ is also in $L^p(\bR)$ for $1<p<\infty$, and $\tilde{\tilde{f}}(x) = -f(x)$. Let $f(x)=\ddx \log^+\abs{x}$, that is, $f(x) = 0$ for $\abs{x}<1$ and $f(x) = 1/x$ for $\abs{x}>1$, so $f\in L^p(\bR)$ for $1<p<\infty$. Then
 \begin{align*}
  \tilde f(x) &= \frac{1}{\pi}\left(\int_{-\infty}^{-1}+ \int_{1}^{\infty}\right) \frac{1}{t(x-t)}\,dt\\
  &= \frac{1}{\pi}\left(\int_{-\infty}^{-1}+ \int_{1}^{\infty}\right) \frac{1}{x}\left(\frac{1}{t} + \frac{1}{x-t}\right) \,dt\\
  &= -\frac{1}{\pi x}\log \left|\frac{x+1}{x-1}\right|.
 \end{align*}
It then follows from the property $\tilde{\tilde{f}}(x) = -f(x)$ of the Hilbert transform that
\[
 -\ddx \log^+ \abs{x} = -f(x) = \frac{1}{\pi} \int_{\bR} \frac{1}{x-t}\left(-\frac{1}{\pi t}\log \left|\frac{t+1}{t-1}\right|\right)\,dt,
\]
and so
\[
 p_\mu'(x) = \frac{1}{\pi} \int_{\bR} \frac{-1}{x-t}\left(\frac{1}{\pi t}\log \left|\frac{t+1}{t-1}\right|\right)\,dt + \ddx \log^+ \abs{x} = 0
\]
as a function of real $x$, for $x\neq \pm 1$. A straightforward change of variables shows that $p_\mu(x)=p_\mu(1/x)$, so since $p_\mu(x)$ is constant on $(-1,1)$, it must be constant on $\bP^1(\bR)\setminus \{-1,1,\infty\}$, and hence by Corollary \ref{cor:constant-potential-is-equil} it must be the equilibrium measure, as any finite set is polar.

We will now establish that
\[
 p_\mu(x)=\frac{7\zeta(3)}{2\pi^2}\quad\text{for all}\quad x\in \bR
\]
by computing the required integral for $p_\mu(x)$ at $x=0$. We begin by expanding in terms of series using $-\log x= \sum_{m\geq 1} (1-x)^m/m$ and $\log\frac{1+x}{1-x}=\sum_{n\geq 0} 2x^{2n+1}/(2n+1)$, both series valid for $\abs{x}<1$:
\begin{align*}
 p_\mu(0) &= 2\int_0^1 \frac{-\log x}{\pi^2 x} \log\frac{1+x}{1-x}\,dx\\
&= 2\int_0^1 \frac{1}{\pi^2 x}\sum_{n=0}^\infty\sum_{m=1}^\infty \frac{(1-x)^m}{m}
\cdot \frac{2x^{2n+1}}{2n+1}\,dx\\
&= \frac{4}{\pi^2}\sum_{n=0}^\infty \sum_{m=1}^\infty \frac{m!(2n)!}{m(m+2n+1)!(2n+1)}
\end{align*}
where we can interchange summation and integration as all terms are nonnegative, and we have used the well-known identity for the beta function $\mathrm{B}(s,t)$ to evaluate the Euler integral that arises:
\[
 \mathrm{B}(s,t) = \int_0^1 x^{s-1} (1-x)^{t-1}\,dx = \frac{(s-1)!\,(t-1)!}{(s+t-1)!}\quad \text{for}\quad s,t\in\bN.
\]
We continue our evaluation now by regrouping the terms and recognizing another Euler integral:
\begin{align*}
 p_\mu(0) &=\frac{4}{\pi^2}\sum_{n=0}^\infty \sum_{m=1}^\infty \frac{1}{(2n+1)^2} \cdot \frac{(m-1)!(2n+1)!}{(m+2n+1)!}\\
 &= \frac{4}{\pi^2}\sum_{n=0}^\infty \frac{1}{(2n+1)^2} \sum_{m=1}^\infty \int_0^1 x^{m-1} (1-x)^{2n+1}\,dx\\
 &= \frac{4}{\pi^2}\sum_{n=0}^\infty \frac{1}{(2n+1)^2} \int_0^1 (1-x)^{2n}\,dx\\
 &= \frac{4}{\pi^2}\sum_{n=0}^\infty \frac{1}{(2n+1)^3} = \frac{7\zeta(3)}{2\pi^2}
\end{align*}
where we evaluated the final sum by noting that $\sum_{n\geq 1} 1/(2n)^3 = \zeta(3)/8$, so $\sum_{n\geq 1} 1/(2n+1)^3 = 7\zeta(3)/8$.

Thus we have established that $\mu$ is the equilibrium measure and 
\[
 I(\mu) = \iint_{\bP^1(\bR)^2} -\log \delta(x,y)\,d\mu(x)\,d\mu(y) = \frac{7\zeta(3)}{2\pi^2}
\]
is the minimal energy $V(\bP^1(\bR))$.
\end{proof}

\begin{proof}[Proof of Theorem \ref{MainThmLocal} (c)]
We now assume our field $L$ is non-archimedean. As above we assume that the absolute value $\abs{\cdot}$ on $L$ has been normalized to coincide with the modulus of additive Haar measure on $L$.  Let $\Ocal_L = \{ \al\in L : \abs{\al}\leq 1\}$ denote the ring of integers of $L$, $\pi$ a uniformizing parameter, and $q = \abs{\Ocal_L/\pi \Ocal_L}$ be the order of its residue field, which is finite since we assumed that $L$ is locally compact. 

First, observe that if $f\in\PGL_2(\Ocal_L)$, then it follows from the definition of the projective metric that $\delta(f(x),f(y))=\delta(x,y)$ for all $x,y\in\PP^1(L)$.  To see that $\mu_L$ is $\PGL_2(\Ocal_L)$-invariant, observe that for any $f\in\PGL_2(\Ocal_L)$ we have
\begin{equation*}
\begin{split}
I(f_*\mu_L) & =\iint_{\PP^1(L)\times\PP^1(L)}-\log\delta(x,y)\,d(f_*\mu_L)(x)\,d(f_*\mu_L)(y) \\
	& = \iint_{\PP^1(L)\times\PP^1(L)}-\log\delta(f(x),f(y))\,d\mu_L(x)\,d\mu_L(y) \\
	& = \iint_{\PP^1(L)\times\PP^1(L)}-\log\delta(x,y)\,d\mu_L(x)\,d\mu_L(y) = I(\mu_L).
\end{split}
\end{equation*}
By the uniqueness of the equilibrium measure, we conclude that $f_*\mu_L=\mu_L$ and therefore that $\mu_L$ is $\PGL_2(\Ocal_L)$-invariant.  To prove uniqueness of the invariant measure, let $\mu$ be an arbitrary $\PGL_2(\Ocal_L)$-invariant unit Borel measure on $\PP^1(L)$.  Given two points $x,x'\in \PP^1(L)$, select $f\in\PGL_2(\Ocal_L)$ such that $x'=f^{-1}(x)$, and we have 
\begin{equation*}
\begin{split}
p_\mu(x') & =\int_{\PP^1(L)}-\log\delta(f^{-1}(x),y)\,d\mu(y) \\
	& =\int_{\PP^1(L)}-\log\delta(x,f(y))\,d\mu(y) \\
	& =\int_{\PP^1(L)}-\log\delta(x,y)\,d(f_*\mu)(y) \\
	& =\int_{\PP^1(L)}-\log\delta(x,y)\,d\mu(y) =p_\mu(x).
\end{split}
\end{equation*}
It follows that $p_\mu(x)$ is constant and hence $\mu=\mu_L$.

To calculate the minimal energy, observe that our normalization means that our uniformizing parameter $\pi$ has absolute value $|\pi|=1/q$. For each $\alpha\in\PP^1(L)$ and $r\geq0$, define the ball $B(\alpha,r)=\{y\in\PP^1(L)\mid\delta(x,y)\leq r\}$. Letting $\alpha_1,\dots,\alpha_q\in\Ocal_L$ be coset represesentatives for the quotient $\Ocal_L/\pi \Ocal_L$, and letting $\alpha_{q+1}=\infty$, we obtain a partition $\PP^1(L)=\amalg_{j=1}^{q+1}B(\alpha_j,1/q)$ of the projective line into $q+1$ balls $B(\alpha_j,1/q)$ of the same $\mu_L$-measure.  Indeed, using the $\PGL_2(\Ocal_L)$-invariance of $\mu_L$, when $1\leq j\leq q$, $B(\alpha_j,1/q)$ is the image of $B(0,1/q)$ under $x\mapsto x+\alpha_j$, and $B(\infty,1/q)$ is the image of $B(0,1/q)$ under $x\mapsto 1/x$.  It follows that $\mu_L(B(0,1/q))=\frac{1}{q+1}$.  For each $n\geq1$, the ball $B(0,1/q)$ is a disjoint union of $q^{n-1}$ balls of the form $B(\alpha,1/q^n)$ for $\alpha\in B(0,1/q)$, and as these balls are all translates of $B(0,1/q^n)$ under maps $x\mapsto x+\alpha$, they have equal measure $\frac{1}{q^{n-1}(q+1)}$. It follows that $\mu_L(B(0,1/q^{n}))=\frac{1}{q^{n-1}(q+1)}$ and so
\begin{equation*}
\begin{split}
\mu_L(\{y\in\PP^1(L)\mid\delta(0,y)=1/q^n\}) & =\mu_L(B(0,1/q^{n})\setminus B(0,1/q^{n+1})) \\
	& =\frac{1}{q^{n-1}(q+1)}-\frac{1}{q^{n}(q+1)} \\
	& =\frac{q-1}{q^n(q+1)}.
\end{split}
\end{equation*}
Since the potential has constant value equal to the minimal energy on $\bP^1(L)$, we can compute the energy by evaluating $p_{\mu_L}$ at a convenient point:
\begin{equation*}
\begin{split}
I(\mu_L) = p_{\mu_L}(0) & =\int_{\PP^1(L)}-\log\delta(0,y)\,d\mu_L(y) \\
	& =\sum_{n\geq1} -\log(1/q^n)\mu_L(\{y\in\PP^1(L)\mid\delta(0,y)=1/q^n\}) \\
	& =\frac{(q-1)\log q}{q+1}\sum_{n\geq1} \frac{n}{q^n} \\
	& =\frac{q\log q}{q^2-1}.\qedhere
\end{split}
\end{equation*}
\end{proof}

\begin{proof}[Proof of Corollary~\ref{DiscreteCor}]
Let $\{Z_{n_k}\}$ be a subsequence of $\{Z_n\}$ for which $D(Z_{n_k})\to \ell$ for some $\ell\in\RR$; we must show that $\ell\geq I(\mu_L)$.  Passing to a further subsequence via Prokhorov's theorem, we may assume without loss of generality that the sequence of measures $\{[Z_n]\}$ converges weakly to some unit Borel measure $\nu$ on $\PP^1(L)$.  By \cite[Lemma 7.54]{BakerRumelyBook} it follows that $\ell\geq I(\nu)$, while $I(\nu)\geq I(\mu_L)$ follows from Theorem~\ref{MainThmLocal}; thus $\ell\geq I(\mu_L)$.

Assume now that $D(Z_n)\to I(\mu_L)$.  In order to show that $[Z_n]\to\mu_L$ weakly, it suffices to show that an arbitrary subsequence of $\{[Z_n]\}$ has a subsequence converging to $\mu_L$.  Let $\{[Z_{n_k}]\}$ be an arbitrary subsequence.  By Prokhorov's theorem, there is further subsequence $\{[Z_{n_{k_j}}]\}$ converging weakly to some unit Borel measure $\nu$ on $\PP^1(L)$.  Again using \cite[Lemma 7.54]{BakerRumelyBook}, we have 
\[
I(\nu)\leq\lim_{j\to\infty}D(Z_{n_{k_j}})=I(\mu_L),
\]
which implies that $I(\nu)=I(\mu_L)$ and thus that $\nu=\mu_L$ by Theorem~\ref{MainThmLocal}. 
\end{proof}

\section{Global height bounds}
We will now state and prove a general lower bound on the the absolute Weil height $h:\PP^1(\Kbar)\to\RR$, in which we work over an arbitrary number field $K$, and we allow for arbitrary inertial and ramification degrees in our splitting at the non-archimedean places. Theorem \ref{SimpleGlobalThm} is a special case of this result:
\begin{thm}\label{GeneralGlobalThm}
Fix a number field $K$, a set of places $S$ of $K$, and a choice of Galois extension $L_v/K_v$ for each $v\in S$, taking $L_v=\bR$ if $v\mid\infty$. Let $L_S$ be the field of all algebraic numbers for which all $K$-Galois conjugates lie in $L_v$ for each $v\in S$. Then
\[
 \liminf_{\al\in L_S} h(\al) \geq \sum_{\substack{v\in S\\ v\nmid\infty}} \frac{N_v}{2}\cdot \frac{q_v^{f_v} \log p_v}{e_v(q_v^{2f_v} - 1)} + \sum_{\substack{v\in S\\ v\mid\infty}} N_v\cdot \frac{7\zeta(3)}{4\pi^2},
\]
where $N_v=[K_v:\bQ_v]/[K:\bQ]$, $\zeta(3)=\sum_{n\geq 1} n^{-3}$, and for $v\nmid\infty$, $e_v=e(L_v/K_v)$ and $f_v=f(L_v/K_v)$ denote the ramification and inertial degrees of $L_v/K_v$, respectively, $q_v$ denotes the order of the residue field of $K_v$, and $p_v$ is rational prime above which $v$ lies.
\end{thm}
\begin{proof}
First, let us note that if $S$ is infinite, we can take a limit over increasing finite subsets of $S$, and the general result will follow, so we may as well assume that $S$ is finite in our proof. We choose absolute values for each $v$ of $K$, normalized so that $\abs{\cdot}_v = \norm{\cdot}_v^{N_v}$ where $\norm{\cdot}_v$ extends the usual absolute of $\bQ$ over which it lies and $N_v=[K_v:\bQ_v]/[K:\bQ]$ as in the theorem statement. Note that if $\abs{\cdot}$ is the absolute value of $K_v$ normalized to coincide with the additive Haar measure of $K_v$, then
$
 \abs{\cdot} = \abs{\cdot}_v^{[K:\bQ]}
$.
With our choice of normalization, the absolute logarithmic Weil height for $\al\in K$ is given by
\[
 h(\al) = \sum_{v\in M_K} \log^+ \abs{\al}_v.
\]

Let $\{\al_k\}_{k=1}^\infty$ denote a sequence of algebraic numbers such that
\[
 \lim_{k\ra\infty} h(\al_{k}) = \liminf_{\al\in L} h(\al),
\]
and let $n_k$ denote the number of $K$-Galois conjugates of $\al_{k}$, which we denote $\al^{(1)}_k,\ldots, \al^{(n_k)}_{k}$. Since the $\al_{k}$ have bounded height, it follows from Northcott's theorem that $n_k\ra\infty$ as $k\ra\infty$. For a place $v$ of $K$ and for a finite set of points $Z\subset \bP^1(K_v)$ with size $N$, let us define the \emph{local discrepancy} of $\al_k$ to be the quantity
\[
D_v(\al_k)=\frac{1}{n_k(n_k-1)}\sum_{1\leq i\neq j\leq n_k}-\log\delta_v(\alpha_k^{(i)},\alpha_k^{(j)}),
\]
where as before $\delta_v$ is the spherical metric in the $v$-adic absolute value $\abs{\cdot}_v$ as normalized above. 

It follows from Corollary \ref{DiscreteCor} applied to the local field $K_v$ that we have
\begin{equation}\label{eqn:main-estimate}
 \liminf_{k\ra\infty} D_v(\al_k)\geq \begin{cases}
    \displaystyle N_v\cdot \frac{q_v^{f_v} \log p_v}{e_v(q_v^{2f_v} - 1)} & \text{if }v\nmid\infty\text{ and }v\in S\\
   \displaystyle  N_v\cdot \frac{7\zeta(3)}{2\pi^2} & \text{if }v\mid\infty\text{ and }v\in S
  \end{cases}
\end{equation}
where again for non-archimedean $v$, $p_v$ denotes the rational prime over which $v$ lies, $f_v = f(L_v/K_v)$ and $e_v=e(L_v/K_v)$ denote the inertial and ramification degrees of $L_v/K_v$. For $v\notin S$, it follows from Mahler's inequality (\cite{Mahler}; see also \cite{BakerAverages}) that 
\begin{equation*}\label{eqn:baker-mahler}
 D_v(\al_k) \geq \begin{cases}\displaystyle
                                                    -N_v \cdot \frac{\log n_k}{n_k-1} & \text{if }v\mid \infty,\\
                                                    0 & \text{otherwise}
                                                   \end{cases},
\end{equation*}
and therefore that 
\begin{equation}\label{eqn:main-estimate-not-in-S}
 \liminf_{k\ra\infty} D_v(\al_k)\geq 0\quad\text{for all}\quad v\notin S.
\end{equation}
Now, by the product formula and our choice of normalizations,
\[
 h(\al_k) = \frac{1}{2} \sum_{v\in M_K} D_v(\al_k),
\]
so if we apply \eqref{eqn:main-estimate} and \eqref{eqn:main-estimate-not-in-S} in the above equation we obtain the desired result.
\end{proof}
Using this result, we are now in a position to prove Corollary \ref{cor:global-equid}. We note that this corollary can trivially be generalized as in Theorem \ref{GeneralGlobalThm} above, however, for simplicity of notation we will prove the simpler statement here.
\begin{proof}[Proof of Corollary \ref{cor:global-equid}]
Using the notation of the previous proof above, we let $D_p(\al_n)$ denote the $p$-adic local discrepancy between the Galois conjugates of $\al_n$. It then follows from our hypothesis, together with equations \eqref{eqn:main-estimate-not-in-S} and \eqref{eqn:main-estimate} from above, that $D_p(\al_n) \ra I(\mu_{\bQ_p})$ for each place $p\in S$. It now follows from Corollary \ref{DiscreteCor} that the probability measures $[\al_n]$ supported equally on the Galois conjugates of $\al_n$ converge weakly to the measures $\mu_{\bQ_p}$ from Theorem \ref{MainThmLocal}.
\end{proof}

\bibliographystyle{abbrv} 
\bibliography{bib}        

\end{document}